\documentclass{article}
\usepackage{amsmath,amsthm,amssymb,amscd}

\newcommand{\ga}{\alpha}
\newcommand{\gb}{\beta}
\renewcommand{\gg}{\gamma}
\newcommand{\gd}{\delta}
\newcommand{\gw}{\omega}

\newcommand{\gS}{\Sigma}
\newcommand{\gs}{\sigma}
\newcommand{\eps}{\varepsilon}

\newcommand{\R}{\mathbb{R}}

\newcommand{\cantor}{2^\gw}
\newcommand{\baire}{\gw^\gw}
\newcommand{\bintree}{2^{<\gw}}

\newcommand{\dotkgen}{\dot K_{\mathit{gen}}}

\newcommand{\coll}{\mathrm{Coll}}

\newcommand{\rng}{\mathrm{rng}}
\newcommand{\power}{\mathcal{P}}
\newcommand{\pioneoneonsigmaoneone}{${\mathbf{\Pi}}^1_1$ on ${\mathbf{\gS}}^1_1$}

\newtheorem{theorem}{Theorem}[section]

\newtheorem{claim}[theorem]{Claim}
\newtheorem{corollary}[theorem]{Corollary}
\newtheorem{fact}[theorem]{Fact}
\newtheorem{proposition}[theorem]{Proposition}

\theoremstyle{definition}
\newtheorem{definition}[theorem]{Definition}
\newtheorem{example}[theorem]{Example}
\newtheorem{question}[theorem]{Question}

\title{Cardinal invariants of closed graphs\footnote{2000 AMS subject classification 03E15, 03E17, 05C15.}}

\author{
Francis Adams
\\University of Florida
\and
Jind{\v r}ich Zapletal
\thanks{The authors were supported by NSF grant DMS  DMS 1161078.}
\\University of Florida
}

\begin{document}
\maketitle

\begin{abstract}
We study several cardinal characteristics of closed graphs $G$ on compact metrizable spaces. In particular, we address the question when it is consistent for the bounding number to be strictly smaller than the smallest size of a set not covered by countably many compact $G$-anticliques. We also provide a descriptive set theoretic characterization of the class of analytic graphs with countable coloring number.
\end{abstract}

\section{Introduction}

The theory of Borel and analytic graphs on Polish spaces is currently a fast growing field \cite{miller:graphs, marks:combinatorics}. In this paper, we contribute to the study of cardinal invariants associated with such graphs. Consider the following invariant:

\begin{definition}
Let $G$ be a graph on a Polish space $X$.

\begin{enumerate}
\item A set $A\subset X$ is a $G$-\emph{anticlique} if no two distinct points of $A$ are $G$-connected;
\item $\kappa(G)$ is the minimum cardinality of a subset of $X$ which is not covered by countably many compact $G$-anticliques.
\end{enumerate}

\noindent If the whole space $X$ is covered by countably many compact anticliques, then let $\kappa(G)=\infty$. 
\end{definition}

Clearly, $\kappa(G)$ is just the uniformity of the $\gs$-ideal generated by compact $G$-anticliques. We consider the problem of comparing the invariant $\kappa(G)$ for various closed graphs $G$ to the standard cardinal invariant $\mathfrak{b}$, the minimum cardinality of a subset of $\baire$ which cannot be covered by countably many compact subsets of $\baire$. While the problem may sound somewhat arbitrary, it in fact connects in an elegant way with various known combinatorial and descriptive set theoretic problems.

\begin{question}
\label{motivatingquestion}
Characterize those closed graphs $G$ for which the inequality $\mathfrak{b}<\kappa(G)$ is consistent with ZFC.
\end{question}

\noindent In order to resolve this question, we introduce another cardinal invariant, the \emph{loose number} of a topological graph (Definition~\ref{loosedefinition}), and prove the main result of the paper:

\begin{theorem}
\label{main1theorem}
Let $G$ be a closed graph on a compact metrizable space $X$. If $G$ has countable loose number, then in some generic extension
$\mathfrak{b}<\kappa(G)$ holds.
\end{theorem}

It may seem that the theorem just replaces one difficult concept with another. However, plenty of information is available on the loose number. It sits neatly between the known combinatorial characteristics of the graph, the chromatic and coloring numbers by Theorem~\ref{betweentheorem}. This immediately yields many informative examples such as locally countable graphs and acyclic graphs; the pre-existing work of \cite{schmerl:list, komjath:list} provides some other natural examples connected with Euclidean spaces.

Fully characterizing the closed or analytic graphs with countable chromatic or loose numbers seems to be a very difficult problem. However, in the case of coloring number, there is a full characterization and a minimal analytic graph of uncountable coloring number:

\begin{theorem}
\label{main2theorem}
There is a closed graph $G_1$ on a Polish space such that for every analytic graph $G$ on a Polish space $X$, exactly one of the following occurs:

\begin{enumerate}
\item $G$ has countable coloring number;
\item there is a continuous injective homomorphism of $G_1$ to $G$.
\end{enumerate}
\end{theorem}

As for the anatomy of the paper, in Section~\ref{singlesection} we introduce the single step forcing to increase the cardinal invariant $\kappa(G)$; Section~\ref{combinatoricssection} connects its forcing properties with combinatorial properties of the graph $G$. Section~\ref{iterationsection} deals with the rather thorny question of iterating the single step forcing. In Section~\ref{examplesection}, we discuss numerous concrete examples. Section~\ref{dichotomysection} contains the proof of the dichotomy theorem for the coloring number of analytic graphs.

We use the standard set theoretic notation of \cite{jech:newset}. For a subset $A$ of a topological space $X$, the symbol $\bar A$ stands for its closure. If $t\in\bintree$ is a finite binary string, then the symbol $[t]$ stands for the basic clopen set $\{x\in\cantor\colon t\subset x\}$ of the Cantor space. The phrase ``large enough structure'' identifies the collection of all sets whose transitive closure has size $<2^{2^{\mathfrak{c}}}$, equipped with the membership relation. Our graphs are non-oriented and do not contain multiplicities or loops; i.e.\ a graph $G$ on a set $X$ is a symmetric relation on $X$ which has empty intersection with the diagonal. If the set $X$ is equipped with a Polish topology, we say that the graph is closed, analytic etc. if it is a closed or analytic relation of the Polish space $(X\times X)\setminus$the diagonal with the topology inherited from the product. An orientation of the graph $G$ is an antisymmetric relation $o$ on $X$ whose symmetrization is equal to $G$. The $o$-outflow of any element $x\in X$ is the set $\{y\in X\colon \langle x, y\rangle\in o\}$.

Many of the results of the present paper appeared in the first author's Ph. D. thesis \cite{adams:thesis}.

\section{Single step forcing}
\label{singlesection}

A forcing notion representing a natural try to increase the cardinal invariant $\kappa(G)$ has been known for quite some time to several authors \cite[Definition 3.3]{geschke:graphs}:

\begin{definition}
Let $X$ be a Polish space and $G$ a closed graph on it. The poset $P_G$ consists of all pairs $p=\langle a_p, o_p\rangle$ where
$a_p\subset X$ is a finite $G$-anticlique and $o_p\subset X$ is an open set containing $a_p$ as a subset. The ordering is defined by $q\leq p$ just in case $a_p\subset a_q$ and $o_q\subset o_p$.
\end{definition}

The poset $P_G$ has a canonical generic object: the closure $\dotkgen$ of the union of the sets $a_p$ where $p$ ranges over all conditions in the generic filter.

\begin{proposition}
Let $X$ be a Polish space and $G$ a closed graph on it. Then $P_G\Vdash\dotkgen\subset X$ is a compact $G$-anticlique.
\end{proposition}

\begin{proof}
Let $d$ be a compatible metric for the space $X$. To show that the set $\dotkgen\subset X$ is forced to be compact, it will be enough to show that it is totally bounded. For each real number $\eps>0$ let $D_\eps=\{q\in P_G\colon$there is a number $n$ such that there is no collection of $n$ many points of $\bar o_q$ which are pairwise at distance greater than $\eps$ from each other$\}$.

\begin{claim}
The set $D_\eps\subset P_G$ is open dense.
\end{claim}

\begin{proof}
It is immediate that the set $D_\eps$ is open. For the density, given any condition $p\in P_G$, for each $x\in a_p$ select an open set $o_x$ containing $x$ of $d$-diameter $<\eps$, and let $q=\langle a_p, o_p\cap\bigcup_{x\in a_p}o_x\rangle\leq p$.
To see that the condition $q$ belongs to the set $D_\eps$, note that every set of points of pairwise distance $>\eps$ which is a subset of $\bar o_q$ can have size at most $|a_p|$.
\end{proof}

To show that $\dotkgen$ is forced to be totally bounded, note that for each $\eps>0$ there must be a condition $p\in D_\eps$ in the generic filter. Such a condition clearly forces that $\dotkgen\subset\bar o_q$ and therefore $\dotkgen$ cannot contain any infinite collection of points which are pairwise at distance greater than $\eps$ from each other.

Now we need to show that $\dotkgen$ is forced to be a $G$-anticlique. For every real number $\eps>0$ write $D_\eps=\{p\in P_G\colon$ any two points in $\bar o_p$ which are $G$-related must be at a distance less than $\eps$ of each other$\}$.

\begin{claim}
$D_\eps$ is open dense in $P_G$.
\end{claim}

\begin{proof}
It is immediate that the set $D_\eps$ is open. For the density, given any condition $p\in P_G$, use the fact that the graph $G$ is closed to find open neighborhoods $o_x$ of each point $x\in a_p$ which are pairwise disjoint, of diameter $<\eps$, and such that $x\neq y\in a_p$ implies $(o_x\times o_y)\cap G=0$. It is not difficult to see that the condition $q=\langle a_p, o_p\cap\bigcup_{x\in a_p}o_x\rangle\leq p$ belongs to the set $D_\eps$.
\end{proof}

Now suppose that $H\subset P_G$ is a generic filter and $x\neq y\in \dotkgen$ are distinct points, at a distance $>\eps$ from each other for some positive rational $\eps$. By the claim and a genericity argument, there is a condition $p\in P$ in the filter $H$ which belongs to the set $D_\eps$. Then, $\dotkgen\subset\bar o_p$, and by a Mostowski absoluteness argument between $V$ and $V[H]$, no two points of $\bar o_p$ (in particular, $x$ and $y$) which are at a distance greater than $\eps$ from each other can be $G$-connected. Since the points $x, y$ were arbitrary, this shows that the set $\dotkgen$ is forced to be a $G$-anticlique.
\end{proof}

The poset $P_G$ is uniquely qualified to resolve our motivating Question~\ref{motivatingquestion}. This follows from the following theorem, together with the absoluteness and iteration results of Section~\ref{iterationsection}:

\begin{theorem}
\label{uniquenesstheorem}
Let $G$ be a closed graph on a compact metrizable space $X$. If $\mathfrak{b}<\kappa(G)$, then the poset $P_G$ adds no dominating reals.
\end{theorem}

\begin{proof}
Suppose towards a contradiction that $p\Vdash\dot z\in\baire$ modulo finite dominates all ground model elements of $\baire$ for some condition $p\in P_G$ and a $P_G$-name $\dot z$ for an element of $\baire$. Let $F\subset\baire$ be an unbounded set of size $\mathfrak{b}$. Let $M$ be an elementary submodel of a large enough structure of size $\mathfrak{b}$ containing $p, \dot z, F$ as elements and $F$ as a subset. Let $b$ be a countable set of compact $G$-anticliques such that $X\cap M\subset\bigcup b$. Let $N$ be a countable elementary submodel of a large structure containing $p, \dot z, b$ as elements. Let $y\in F$ be a function which is not modulo finite dominated by any element of $\baire\cap N$. By the elementarity of the model $M$, there is a condition $q\in P_G\cap M$ and a natural number $m$ such that $q\Vdash \forall k>m\ \dot z(k)>\check y(k)$.

Now, choose a one-to-one enumeration $a_q=\{ x_i\colon i\in j\}$ of the anticlique in the condition $q$. There must be $P_i, O_i, K_i$ for $i\in j$ such that

\begin{itemize}
\item $P_i, O_i$ are basic open subsets of $X$ and $K_i\in b$ are compact $G$-anticliques;
\item $o_q\supset P_i\supset \bar O_i$, the sets $\bar O_i$ are pairwise disjoint and $G$-disconnected;
\item $x_i\in O_i\cap K_i$.
\end{itemize}  

Note that this sequence of objects belongs to the model $N$. Now, for each number
$k>m$, consider the set $c_k$ of those numbers $l\in\gw$ such that for some condition $r_l\in P_G$,  $r_l\Vdash\dot z(k)>l$ holds, $\bigcup_i P_i\subset o_{r_l}$, and the set $a_{r_l}$ can be listed as $\{x_i^l\colon i\in j\}$ so that for each $i\in j$, $x_i^l\in O_i\cap K_i$ holds.
The key claim:

\begin{claim}
The set $c_k$ is finite.
\end{claim}

\begin{proof}
Suppose towards a contradiction that the set $c_k$ is infinite, and for each number $l\in c_k$ select a condition $r_l\in P_G$ witnessing the membership in $c_k$, and let $x_i^l$ for $i\in j$ denote the unique element of $a_r\cap O_i\cap K_i$.
Using the compactness of the space $X$ and thinning out the set $c_k$ if necessary, we may assume that the points $x_i^l$ converge to some $\hat x_i\in X$ for each $i\in j$. Note that as the $G$-anticliques $K_i$ are closed, $\hat x_i\in K_i$ holds.

Now, consider the condition $s\in P_G$ given by the demands $a_s=\{\hat x_i\colon i\in j\}$ and $o_s=\bigcup_{i\in j} P_i$. This is indeed a condition: the set $a_s$ is a $G$-anticlique by the choice of the basic open sets $O_i$. We will reach the contradiction by showing that $s$ forces infinitely many of the conditions $r_l$ into the generic filter. This is of course impossible since then $s$ forces that there is no value that the name $\dot z(k)$ can attain.

Suppose that $t\leq s$ is a condition and $\hat l$ is some natural number; we must find a number $l>\hat l$ and a lower bound of $t$ and $r_l$. To find this number $l$, use the fact that the graph $G$ is closed to find respective neighborhoods $\hat O_i$ of points $\hat x_i$ for $i\in j$ such that $\hat O_i\subset o_t$ and for all points $x\neq \hat x_i$, $x$ has no $G$-neighbors in the set $\hat O_i$. Find a number $l>\hat l$ so large that for each $i\in j$ the points $x_i^l$ belong to the sets $\hat O_i$ for each $i\in j$. Observe that the set $a_t\cup a_{r_l}$ then must be a $G$-anticlique: if $i\in j$ and $x\neq \hat x_i$ is a point in $a_t$, then $x_i^l$ is not $G$-connected to $x$ since $x_i^l\in\hat O_i$, and $x_i^l$ is not $G$-connected to $\hat x_i$ since both belong to the same compact anticlique $K_i$. It now follows immediately that $\langle a_t\cup a_{r_l}, o_t\cap o_{r_l}\rangle$ is a condition in the poset $P_G$ and a lower bound of $t, r_l$.
\end{proof}

Now note that the sequence $\langle c_k\colon k>m\rangle$ belongs to the model $N$ by elementarity, and so the model $N$ contains some function $z\in\baire$ such that for each $k>m$ returns a value larger than $\max(c_k)$. Note also that $y(k)\in c_k$ holds as the condition $q$ witnesses the membership. This means that the function $z\in N$ modulo finite dominates the function $y$, contradicting the choice of $y$.
\end{proof}

\section{Combinatorics}
\label{combinatoricssection}

Theorem~\ref{uniquenesstheorem} does not shed any light on how to actually evaluate the critical forcing properties of the poset $P_G$ in any specific case. It turns out though that the forcing properties of the poset $P_G$ faithfully reflect certain combinatorial cardinal invariants of the graph $G$. In order to state the interesting correspondence theorem, we must introduce the relevant invariants and forcing features.

\begin{definition}
\textnormal{\cite{hajnal:chromatic}}
Let $G$ be a graph on a set $X$. The \emph{chromatic number} $\chi(G)$ of the graph $G$ is the smallest cardinality of a collection of $G$-anticliques covering the space $X$.
\end{definition}

\noindent If the set $X$ is equipped with a topology and the graph $G$ is open then the closure of any anticlique is again an anticlique. However, in most interesting closed graphs, anticliques cannot be in general enclosed by closed, Borel, or analytic anticliques. Thus, constructing a cover of the space by $G$-anticliques becomes a process in which the axiom of choice must be considered.

\begin{definition}
\label{loosedefinition}
Let $G$ be a graph on a topological space $X$. A $G$-\emph{loose set} is a set $A\subset X$ such that for every point $x\in X$ there is an open neighborhood $O$ of $x$ containing no elements of the set $A$ which are $G$-connected to $x$. The \emph{loose number} $\lambda(G)$ of the graph $G$ is the smallest cardinality of a collection of $G$-loose sets covering the space $X$.
\end{definition}

\noindent A rather primitive example of a $G$-loose set is a closed $G$-anticlique. Not every $G$-loose set needs to be an anticlique, but every $G$-loose set is a union of countably many anticliques. On the other hand, one can find graphs in which there are $G_\gd$-anticliques which are not unions of countably many $G$-loose sets, see Example~\ref{easyexample}.

The last relevant cardinal invariant of a graph $G$ is the coloring number. There are several equivalent ways to define it. To shorten the arguments, we use a definition which may give different values from others in the case these values are finite. This wrinkle is inconsequential for this paper.

\begin{definition}
\textnormal{\cite{hajnal:chromatic}}
Let $G$ be a graph on a set $X$. The \emph{coloring number} $\mu(G)$ of $G$ is the smallest cardinal $\kappa$ such that there is an orientation of the edges of $G$ such that the outflow of each vertex has size $<\kappa$.
\end{definition}

\noindent A good example of a graph with countable coloring number is a locally countable graph. To construct the orientation, choose an enumeration of each connected component by natural numbers, and within the component, orient the edges towards the vertex with a smaller index in that enumeration. Another example of a graph with countable coloring number (in fact, coloring number equal to $2$) is an acyclic graph. In each of it connected components, choose a single vertex and orient the edges within the component towards the chosen vertex.

While the chromatic and coloring numbers have been studied for many years, the loose number is a new concept. Note that unlike the chromatic and coloring numbers it depends on the topology of the underlying space. The following theorem shows the important implications among the three concepts.

\begin{theorem}
\label{betweentheorem}
Let $G$ be a graph on a Polish space $X$. Then $\mu(G)\leq\aleph_0$ implies $\lambda(G)\leq\aleph_0$, which in turn implies $\chi(G)\leq\aleph_0$.
\end{theorem}

The implications cannot be reversed even in the case of closed graphs on Polish spaces, which is clear from the examples in Section~\ref{examplesection}. The theorem makes the loose number look suspiciously close to the list-chromatic number, but this is a red herring: Corollary~\ref{listcorollary} below shows that for analytic graphs $G$, the list-chromatic number is countable just in case its coloring number is countable, and therefore the list-chromatic number is not useful for the present discussion.

\begin{proof}
Suppose first that the coloring number of the graph $G$ is countable. Let $o$ be an orientation of the edges of $G$ such that the $o$-outflow of any point $x\in X$ is finite. To each point $x\in X$ assign a basic open set $f(x)\subset X$ containing $x$ such that its closure contains none of the points in the finite $o$-outflow of $x$. For every basic open set $O\subset X$ let $A^O=\{x\in X\colon f(x)=O\}$.

\begin{claim}
The set $A^O$ is $G$-loose.
\end{claim}

\begin{proof} 
If this failed, there would be a point $y\in X$ such that each neighborhood of $y$ contains some point of $A^O$ not equal to $y$ and $G$-connected to $y$. In particular, $y\in\bar A^O\subset\bar O$. Let $P$ be an open neighborhood of $y$ containing none of the points in the finite $o$-outflow of $y$, and let $x\in A^O\cap P$ be a point $G$-connected to $y$. The edge $\{ x, y\}\in G$
cannot be oriented towards the point $x$ by the choice of the neighborhood $P$. It also cannot be oriented towards the point $y$ since $y\in\bar O$ and no points in $\bar O$ are in the outflow of the point $x$. A contradiction.
\end{proof}

\noindent Clearly, the $G$-loose sets $A^O$, as $O$ varies over some fixed countable basis of the space $X$, cover the whole  space, and so the loose number of the graph $G$ is countable.

Now, suppose that the loose number of $G$ is countable, and let $X=\bigcup_nA_n$ be a cover of the space $X$ by countably many $G$-loose sets. For each $n\in\gw$ and each point $x\in X$, let $f_n(x)\subset X$ be some basic open set containing $x$ and no points of $A_n$ which are $G$-connected to $x$. For every $n\in\gw$ and every basic open set $O\subset X$, let $A_{n}^O=\{x\in A_n\colon f_n(x)=O\}$.

\begin{claim}
The set $A_{n}^O$ is a $G$-anticlique.
\end{claim}

\begin{proof}
Suppose that $x\neq y$ are distinct points in the set $A_n^{O}$. Both of the points belong to both $A$ and $O$. By the definition of the function $f_n$, no elements of $A\cap O$ are connected to $x$, in particular $y$ is not connected to $x$. The claim follows.
\end{proof}

\noindent Clearly, for each number $n\in\gw$ the sets $A_{n}^O$ cover the set $A_n$ as $O$ varies over some fixed countable basis of the space $X$. Thus, $X=\bigcup_{n,O}A_{n}^O$ is a countable cover of the whole space by countably many $G$-anticliques and so the chromatic number of $G$ is countable.
\end{proof}

The forcing properties of the poset $P_G$ connected to the combinatorial concepts listed above are the following:

\begin{definition}
Let $\langle P, \leq\rangle$ be a partial ordering. Let $A\subset X$. The set $A$ is \emph{centered} if for every finite set $b\subset A$ there is a condition $q\in P$ such that for every $p\in b$, $q\leq p$. The poset $P$ is $\gs$-\emph{centered} if it can be written as a union of countably many centered sets.
\end{definition}

\begin{definition}
Let $\langle P, \leq\rangle$ be a partial ordering. Let $A\subset X$. The set $A$ is \emph{liminf centered} if for every sequence
$\langle p_i\colon i\in\gw\rangle$ of elements of $A$, there is a condition $q\in P$ such that for every $r\leq p$ the set $\{i\in\gw\colon p_i$ is compatible with $r\}$ is infinite. The poset $P$ is $\gs$-\emph{liminf-centered} if it can be written as a union of countably many liminf centered sets.
\end{definition}

\noindent Both $\gs$-centeredness and $\gs$-liminf-centeredness imply c.c.c.\ since no centered set can contain two incompatible elements, and no liminf-centered set can contain an infinite antichain. In general, there are no implications between $\gs$-centeredness and $\gs$-liminf-centeredness: for example, the Hechler forcing is $\gs$-centered but not $\gs$-liminf-centered, while the random forcing is $\gs$-liminf-centered but not $\gs$-centered.

\begin{theorem}
\label{equitheorem}
Let $G$ be a closed graph on a compact metrizable space $X$.

\begin{enumerate}
\item \textnormal{\cite[Proposition 1]{todorcevic:examples}} $P_G$ is c.c.c.\ iff $G$ contains no perfect clique;
\item $P_G$ is $\gs$-centered iff $\chi(G)\leq\aleph_0$;
\item $P_G$ is $\gs$-liminf-centered iff $\lambda(G)\leq\aleph_0$.
\end{enumerate}
\end{theorem}

\begin{proof}
To see the left-to-right implication of (2), let $P_G=\bigcup_nA_n$ be a countable union of centered sets, and for every number $n\in\gw$ let $B_n=\{x\in X\colon \langle \{x\}, X\rangle\in A_n\}$. It is immediate that $X=\bigcup_nB_n$ holds; we must show that each set $B_n$ is a $G$-anticlique. This, however, is immediate since any edge in $B_n$ would result in a pair of incompatible conditions in $P_G$.

To see the right-to-left implication of (2), suppose that  $X=\bigcup_n B_n$ be a countable union of countably many $G$-anticliques.
Say that a finite sequence $t=\langle P_i, O_i, n_i\colon i\in j\rangle$ is \emph{good} if $P_i, O_i$ are basic open subsets of $X$, $\bar O_i\subset P_i$, the sets $\bar O_i$ are pairwise disjoint and $G$-disconnected, and $n_i\in\gw$. For each good sequence $t$, let $A_t\subset P_G$ be the set of those conditions $p=\langle a_p, o_p\rangle$ such that $a_p$ can be listed as $\{x_i\colon i\in j\}$ with $x_i\in \bar O_i\cap B_{n_i}$, and $o_p\supset \bigcup_{i\in j}P_i$. Since there are only countably many good sequences, the proof will be complete if we show that the set $A_t$ is centered.

Indeed, if $b\subset A_t$ is a finite set, then $a=\bigcup_{p\in b}a_b$ is a $G$-anticlique: if $x_0, x_1\in a$ then either both of them belong to the same set $O_i\cap B_{n_i}$ and they are $G$-disconnected as $B_{n_i}$ is an anticlique, or the belong to distinct such sets, and they are again $G$-disconnected since $\bar O_{i_0}$ is $G$-disconnected from $\bar O_{i_1}$ if $i_0\neq i_1$. It follows that the pair $\langle a, \bigcup_{i\in j}P_i\rangle$ is a condition in $P_G$ which is a lower bound of the set $b$.

To see the left-to-right implication of (3), let $P_G=\bigcup_nA_n$ be a countable union of liminf-centered sets, and for every number $n\in\gw$ let $B_n=\{x\in X\colon \langle \{x\}, X\rangle\in A_n\}$. It is immediate that $X=\bigcup_nB_n$ holds; we must show that each set $B_n$ is $G$-loose. Indeed, suppose towards a contradiction that there is a sequence $\langle x_m\colon m\in\gw\rangle$ of points in some set $B_n$ which converge to some point $y\in X$ and at the same time are connected to $y$. Since the set $A_n$ is liminf-centered, there must be a condition $p\in P_G$ which forces infinitely many points on the sequence to belong to the generic compact anticlique $\dotkgen$. However, then $p\Vdash\check y\in\dotkgen$ as well, contradicting the fact that $\dotkgen$ is forced to be a $G$-anticlique.

For the right-to-left implication of (3), suppose that $X=\bigcup_n B_n$ is a countable union of countably many $G$-loose sets.
Say that a finite sequence $t=\langle P_i, O_i, n_i\colon i\in j\rangle$ is \emph{good} if $P_i, O_i$ are basic open subsets of $X$, $\bar O_i\subset P_i$, the sets $\bar O_i$ are pairwise disjoint and $G$-disconnected, and $n_i\in\gw$. For each good sequence $t$, let $A_t\subset P_G$ be the set of those conditions $p=\langle a_p, o_p\rangle$ such that $a_p$ can be listed as $\{x_i\colon i\in j\}$ with $x_i\in \bar O_i\cap B_{n_i}$, and $o_p\supset \bigcup_{i\in j}P_i$. Since there are only countably many good sequences, the proof will be complete if we show that the set $A_t$ is liminf-centered.

To this end, suppose that $\langle p_l\colon l\in\gw\rangle$ is a countable sequence of conditions in the set $A_t$. List $a_{p_l}$ as $\{x_i^l\colon i\in j\}$ so that $x_i^l\in \bar O_i\cap B_{n_i}$. Use the compactness of the space $X$ and thin out the countable sequence of conditions if necessary to make sure that the sequence $\langle x_i^l\colon l\in\gw\rangle$ converges to a point $\hat x_i\in X$, this for each index $i\in j$. Since each of the sets $B_{n_i}$ is $G$-loose, there is a basic open neighborhood $R_i$ of $\hat x_i$ such that no points in $B_{n_i}$ in this neighborhood are connected to $\hat x_i$. Consider the condition $q\in P_G$ given by the following demands: $a_q=\{\hat x_i\colon i\in j\}$ and $o_q=\bigcup_{i\in j}P_i\cap R_i$. Since each point $\hat x_i$ belongs to $\bar O_i$ and the sets $\bar O_i$ are pairwise $G$-disconnected, it is clear that $a_q$ is a $G$-anticlique and so $q$ is indeed a condition in $P_G$. It will be enough to show that $q$ forces the set of all $l$ for which $p_l$ is in the generic filter to be infinite.

To this end, let $r\leq q$ be a condition and $\hat l\in\gw$ be a number. We need to produce a number $l>\hat l$ and a lower bound of the conditions $r, p_l$. To this end, for each $i\in j$ find open neighborhoods $S_i\subset X$ of $\hat x_i$ such that no point $x\in a_r$ with $x\neq \hat x_i$ has any $G$-neighbors in the set $S_i$. Find $l>\hat l$ so large that for each $i\in j$, the point $x_i^l$ belongs to $S_i\cap R_i$. Now note that the set $a_r\cup a_{p_l}$ is a $G$-anticlique: a point $x_i^l$ is not $G$-connected to any point $x\neq\hat x_i$ in $a_r$ because $x_i^l\in S_i$ holds, and it is not $G$-connected to $\hat x_i$ either since $x_i^l\in R_i\cap B_{n_i}$ holds. It follows that the pair $\langle a_r\cup a_{p_l}, o_r\rangle\in P_G$ is the requested lower bound of conditions $r$ and $p_l$.
\end{proof}

For the purposes of this paper, the $\gs$-liminf-centered property of posets has the following central implication:

\begin{theorem}
Let $P$ be a $\gs$-liminf-centered poset. Then $P$ does not add a dominating real. In fact, $P$ preserves unboundedness of all subsets of $\baire$.
\end{theorem}

\begin{proof}
Suppose towards a contradiction that $F\subset\baire$ is an unbounded set, $p\in P$ and $\dot z$ is a $P$-name for an element of $\baire$ such that $p\Vdash\dot z$ modulo finite dominates all elements of $F$. Let $P=\bigcup\{ A_n\colon n\in\gw\}$ be a union of liminf-centered sets. Let $M$ be a countable elementary submodel of a large enough structure containing $F, p, \dot z$ and $A_n$ for $n\in\gw$ as elements. Let $y\in F$ be a function which is not modulo finite dominated by any element of $\baire\cap M$, and let $q\leq p$ be a condition and $m\in\gw$ be a number such that $q\Vdash\forall k>m\ \dot z(k)>\check y_k$.

 Let $n\in\gw$ be such that $q\in A_n$. For each number $k\in\gw$, let $c_k\subset\gw$ be the set of all numbers $l$ such that there is a condition $r_{lk}\in A_n$ such that $r_{lk}\Vdash\dot z(k)>\check l$. The set $c_k$ must be finite, since otherwise by the liminf-centeredness of the set $A_n$ there would be a condition $r\in P$ which forces infinitely many of the conditions $r_{lk}$ into the generic filter, which means that $r$ forces that there is no value the name $\dot z(k)$ can attain. Now, the sequence $\langle c_k\colon k\in\gw\rangle$ belongs to the model $M$ by elementarity, and so is the function $\hat y\in\baire$ given by $\hat y(k)=\max(c_k)$. By the definitions, for all $k>m$, the number $y(k)$ belongs to the set $c_k$. This means that $\hat y$ dominates $y$ modulo finite, contradicting the choice of the function $y$.
\end{proof}

The reward for all the work in this section is the following corollary, which shows that closed graphs with countable loose number behave well in an important respect:

\begin{corollary}
Suppose that $G$ is a closed graph on a compact metrizable space $X$. If $X$ can be written as a countable union of $G$-loose sets, then the poset $P_G$ is c.c.c.\ and adds no dominating real.
\end{corollary}

\section{Iteration}
\label{iterationsection}

This section is devoted to the problem of iterating the poset $P_G$ without adding dominating reals, which we find quite tricky. The main difficulty is that in the absence of a characterization of countable loose number of closed or analytic graphs in descriptive theoretic terms, even if the graph $G$ has countable loose number in the ground model, there is no guarantee that it will maintain this property in the intermediate generic extensions arising in the iteration. Failing that, it could occur that the poset $P_G$ destroys some unbounded sequences in the intermediate extensions, and the iteration adds a dominating real after all. We could not find any example of such a situation as all closed graphs of countable loose number we know of possess this property in all generic extensions. Still, the difficulty forced us to consider a rather unlikely workaround. It starts with the following natural definition:

\begin{definition}
A pair $\langle P, \leq\rangle$ is a Borel partial order if the relation $\leq$ is a partial ordering on $P$, there is a Polish space $X$ such that $P\subset X$ is analytic, and the relations $\leq$, compatibility, and incompatibility are analyticl subsets of $X\times X$.
\end{definition}

\noindent As a particularly relevant example, if $G$ is a closed graph on some Polish space $X$, then the poset $P_G$ can be easily viewed as a Suslin partial ordering. Note that every Suslin partial ordering has a natural interpretation in every generic extension, as a straightforward application of the Shoenfield absoluteness shows. An interesting question appears, which of the forcing properties
of the poset $P$ are absolute throughout all forcing extensions? The most important general fact in this direction is the following classical absoluteness theorem:

\begin{fact}
\label{abso1fact}
\textnormal{\cite[Corollary 3.6.9]{bartoszynski:set}}
Let $\langle P, \leq\rangle$ be a Suslin partial order. If in some generic extension $P$ is c.c.c.\ then in all generic extensions $P$ is c.c.c.
\end{fact}

\noindent For our main theorem, the key point is the absoluteness of adding no dominating reals. This is handled by the following:

\begin{theorem}
\label{btheorem}
Let $\langle P, \leq\rangle$ be a Suslin c.c.c.\ partial order. The following are equivalent:

\begin{enumerate}
\item in some generic extension, $P$ adds a dominating real;
\item in every forcing extension, $P$ adds a dominating real.
\end{enumerate}

\noindent Moreover, if $P$ adds no dominating reals then it preserves all unbounded sequences of elements of $\baire$ which are modulo finite increasing, consist of monotonic functions, and have regular uncountable length greater than $\gw_1$.
\end{theorem}

The theorem quickly follows from two propositions of independent interest.

\begin{proposition}
\label{dproposition}
Let $P$ be a Suslin forcing. Suppose that $\kappa>\gw_1$ is a regular cardinal, and $F\colon\kappa\to\baire$ be a $\kappa$-sequence of monotonic functions in $\baire$ increasing and unbounded in the modulo finite domination order. Suppose that $P\Vdash\gs\in\baire$ is a function modulo finite dominating all functions in $\rng(F)$. Then $P\Vdash\gs$ is a dominating real.
\end{proposition}

\begin{proof}
The argument depends on two well-known claims:

\begin{claim}
\label{naclaim1}
Suppose that $Q$ is a partial ordering and $\gs$ is a $Q$-name for a monotonic function in $\baire$, and $P\Vdash\gs$ is not bounded by any ground model element of $\baire$. Then, there are filters $G_0, G_1\subset Q$ separately generic over the ground model such that $\max(\gs/G_0, \gs/G_1)$ modulo finite dominates all elements of the ground model.
\end{claim}

\begin{proof}
Enter a generic extension $V[H]$ in which both $\power(P)\cap V$ and $\baire\cap V$ are countable sets. Let $y\in\baire$ be a monotonic function which modulo finite dominates all functions in the ground model, and let $\{D_n\colon n\in\gw\}$ be an enumeration of all open dense subsets of $P$ in the ground model. By induction on $n\in\gw$ build conditions $p_n, q_n\in P$ and numbers $k_n,l_n\in\gw$ so that

\begin{enumerate}
\item $p_0\geq p_1\geq\dots$ and $q_0\geq q_1\geq\dots$;
\item $p_{n+1}, q_{n+1}\in D_n$;
\item $k_n<l_n<k_{n+1}$;
\item there are infinitely many $m$ for which there is a condition $p\leq p_n$ forcing $\gs(\check k_n)=\check m$,
and there are infinitely many $m$ for which there is a condition $q\leq q_n$ forcing $\gs(\check l_n)=\check m$;
\item $p_{n+1}\Vdash\gs(\check k_n)>\check y(l_n)$ and $q_{n+1}\Vdash\gs(\check  l_n)>\check y(k_{n+1})$.
\end{enumerate}

\noindent To see how this is done, suppose that $p_n, q_n, k_n, l_n$ have been found. First, use item (4) to find a condition $p\leq p_n$ which forces $\gs(\check k_n)>\check y(l_n)$. Strengthen $p$ into a condition $p_{n+1}\in D_n$, and use the fact that $\gs$ is forced to be an unbounded real to find a number $k_{n+1}>l_n$ such that for infinitely many $m\in\gw$ the condition $p_{n+1}$ can be strengthened to force $\gs(k_{n+1})=\check m$. Then, use item (4) again to find a condition $q\leq q_n$ which forces $\gs(\check l_n)>\check y(k_{n+1})$. Strengthen $q$ into a condition $q_{n+1}\in D_n$, and use the fact that $\gs$ is forced to be an unbounded real to find a number $l_{n+1}>k_{n+1}$ such that for infinitely many $m\in\gw$ the condition $q_{n+1}$ can be strengthened to force $\gs(l_{n+1})=\check m$. This concludes the induction step.

In the end, the filters $G_0, G_1\subset Q$ generated by the sets $\{p_n\colon n\in\gw\}$ and $\{q_n\colon n\in\gw\}$ will work as desired by the second and last items of the induction hypothesis. Note that the function $\max(\gs/G_0, \gs/G_1)$ dominates $y$ from $k_0$ on.
\end{proof}

\begin{claim}
\label{naclaim2}
Let $A$ be a set. In some generic extension, there is an elementary embedding $j$ of $V$ into an $\gw$-model $N$ such that $j''A$ is a countable set in $N$.
\end{claim}

\begin{proof}
Just let $Q$ be the poset of all stationary subsets of $[A]^{\aleph_0}$ ordered by inclusion. Whenever $G\subset Q$ is a filter generic over $V$, one can form the generic ultrapower $j\colon V\to N$ modulo $G$. Its target model is an $\gw$-model since the ideal of nonstationary subsets of $[A]^{\aleph_0}$ is $\gs$-complete, and the identity function represents the set $j''A\in N$, which is then countable in $N$ by the \L o\' s theorem.
\end{proof}

Now, suppose that $P$ is a Suslin c.c.c.\ partial ordering. Without loss of generality, assume that its underlying Polish space $X$ is the Baire space. The set $A=\{x\in X^\gw\colon$ if $\rng(x)$ is an antichain in $P$ then it is a maximal antichain$\}$ is coanalytic. Fix a continuous function $f$ with domain $X^\gw$ and range included in the space of all trees on $\gw$ such that for all $x\in X^\gw$, $x\in A$ just in case $f(x)$ is a wellfounded tree.

Suppose that $F\colon\kappa\to \baire$ is a modulo finite increasing sequence of monotonic functions in $\baire$ of regular uncountable length $\kappa>\gw_1$, which is unbounded in $\baire$. Suppose that $\gs$ is a $P$-name for a real that dominates all functions in $\rng(F)$, and argue that $\gs$ is in fact a name for a dominating real. To this end, let $E$ be a function with domain $\kappa$ which for every $\ga\in\kappa$ returns an enumeration $E(\ga)$ of a countable maximal antichain in $P$ consisting of conditions $p\in P$ for which there is $n\in\gw$ such that for each $m>n$ there is  no condition in the antichains in the name $\gs$ which forces $\gs(m)<y_\ga(m)$ and is compatible with $p$. Let $D(\ga)=f(E(\ga))$; these trees are all well-founded, each of some countable rank, and by a counting argument, one can thin down the sequence so that the ranks are all bounded by some countable ordinal $\gb$. This is where we use the assumption that the sequence has length greater than $\gw_1$.

Now, use Claim~\ref{naclaim2} to find a poset $Q$ and a $Q$-name $j$ for an elementary embedding of the ground model to an $\gw$-model $N$ containing the set $j''\kappa$ so that this set is countable in $N$. Let $\dot\gg$ be a $Q$-name for the supremum of $j''\kappa$, which is an $N$-ordinal in $j\kappa$. The $Q$-name $jF(\gg)$ is a name for an unbounded real, since it is forced to modulo finite dominate all points on the sequence. Also, the $Q$-name $jE(\gg)$ is a name for a maximal antichain of $P$. This is somewhat trickier: $jE(\ga)$ is a maximal antichain in the model $N$, but $N$ need not be wellfounded and so may not be correct about maximal antichains of the poset $P$. However, $N$ calculates the tree $jD(\gg)$ correctly and by elementarity it models that it is wellfounded of rank $\gb$. Since $N$ is an $\gw$-model, it follows that $j\gb=j''\gb$ and in the $Q$-extension, the $N$-ordinal $j(\gb)$ is wellfounded, the tree $jD(\ga)$ is wellfounded, and $jE(\gg)$ is a maximal antichain in $P$.

Use Claim~\ref{naclaim1} to find a generic extension $V[H]$ in which there are filters $G_0, G_1\subset Q$ which are separately generic over the ground model such that the point $\max\{jF(\gg)/G_0, jF(\gg)/G_1\}$ modulo finite dominates all elements of the ground model. Let $K\subset P$ be a filter generic over $V[H]$. By the previous paragraph, the point $\gs/K\in\baire$ modulo finite dominates both $jF(\gg)/G_0, jF(\gg)/G_1$ and so is dominating over $V$. At the same time, the filter $K\cap V$ is generic over $V$, since the model $V$ is transitive and therefore correct about maximal antichains of the poset $P$. Thus, we produced a filter on $P$ which is generic over $V$ and at the same time evaluates $\gs$ as a real dominating over $V$.
\end{proof}

\begin{proposition}
\label{cproposition}
Let $P$ be a Suslin c.c.c.\ forcing and let $M$ be a transitive model of set theory containing the code for $P$. Let
$\gs\in M$ be a $P$-name for an element of $\baire$ such that $M\models P\Vdash\gs$ is a dominating real. Then $P\Vdash\gs$ is a dominating real.
\end{proposition}

\begin{proof}
It is easy to show that the statement ``$\gs$ is a $P$-name for a dominating real'' is $\gS^1_2$ in parameter $\gs$. Thus, if the model $M$ contains all countable ordinals, the proposition is an immediate consequence of the Shoenfield absoluteness. However, we want to apply the proposition exactly in the case when $M$ is countable, and more work is needed.

Let $p\in P$ be a condition and $y\in\baire$ be an arbitrary function; we will produce a filter $G\subset P$ generic over $V$ and such that $p\in G$ and $\gs/G$ modulo finite dominates $y$. The proposition then follows by a genericity argument. Let $Q$ be the Hechler forcing. Let $H\subset Q$ be a filter generic over $V$. Since the model $M$ is transitive, it evaluates maximality of antichains in $Q$ correctly and therefore, the filter $H\cap M$ is Hechler-generic over $M$. Let $G\subset P$ be a filter generic over $V[H]$, containing the condition $p$. Since the model $M[H\cap M]$ is transitive, it evaluates the maximality of antichains in $P$ correctly, and therefore, the filter $G\cap M[H\cap M]$ is $P$-generic over $M[H\cap M]$. By the Shoenfield absoluteness applied between the models $M$ and $M[H\cap M]$, $M[H\cap M]\models\gs$ is a $P$-name for a dominating real. Thus, $\gs/G$ must modulo finite dominate the Hechler real, which can be found in the model $M[H\cap M]$. The Hechler real dominates the function $y$. As a result,
$\gs/H$ modulo finite dominates the function $y$ as desired.
\end{proof}

\begin{proof}[Proof of Theorem~\ref{btheorem}]
For the equivalence of (1) and (2), suppose that $V[G]$ and $V[H]$ are two generic extensions and in $V[H]$, there is a condition in $P$ below which $P$ adds a dominating real. We must show that in $V[G]$, there is a condition in $P$ below which $P$ adds a dominating real.

By a downward Loewenheim--Skolem argument, in $V[H]$ there is a countable transitive model of a large fragment of set theory
which satisfies that $P$ below some condition adds a dominating real. By a Shoenfield absoluteness argument, such a model must exist in the model $V[G]$. Proposition~\ref{cproposition} applied in $V[G]$ then shows that in $V[G]$, $P$ below some condition adds a dominating real.

The last sentence of the theorem is just a specialization of Proposition~\ref{dproposition}.
\end{proof}

\noindent The conclusion of Theorem~\ref{btheorem} is rather odd. In the presence of Woodin cardinals, the demand on the length of the unbounded sequence disappears as the embedding in Claim~\ref{naclaim2} can be found so that its target model is transitive. Still, the conclusion is good enough to dovetail with the following iteration theorem:

\begin{fact}
\label{iterationfact}
\textnormal{\cite[Lemma 6.5.3]{bartoszynski:set}}
Let $F\subset\baire$ be a unbounded family such that every countable subset of $F$ has an upper bound in $F$ in the modulo finite domination ordering. Let $\langle P_\ga\colon\ga\leq\lambda, \dot Q_\ga\colon\ga\in\lambda\rangle$ be a finite support iteration of c.c.c.\ forcings such that each iterand preserves the unboundedness of $F$. Then the whole iteration preserves the unboundedness of $F$.
\end{fact}

Now we are ready to prove Theorem~\ref{main1theorem}. Suppose that $G$ is a closed graph on a compact metrizable space $X$ with countable loose number. The poset $P_G$ is $\gs$-liminf-centered by Theorem~\ref{equitheorem}, therefore c.c.c.\ and does not add dominating reals. By Fact~\ref{abso1fact} and Theorem~\ref{btheorem}, the poset $P_G$ remains c.c.c.\ and preserves all increasing unbounded sequences of regular length $>\gw_1$ in all forcing extensions. Pick regular cardinals $\gw_1<\kappa<\lambda$. First, use a c.c.c.\ forcing of size $\kappa$ to add a sequence $F\colon\kappa\to\baire$ which consists of increasing functions in $\baire$ and it is modulo finite increasing and unbounded. Then, iterate the poset $P_G$ $\lambda$-many times with finite support. By Fact~\ref{abso1fact}, the iterands are c.c.c.\ and so is the whole iteration. By Theorem~\ref{btheorem}, the iterands preserve the unboundedness of the sequence $F$, and so does the whole iteration by Fact~\ref{iterationfact}.

In the resulting model, the bounding number $\mathfrak{b}$ is $\leq\kappa$, since the sequence $F$ is unbounded there. It is also true that $\kappa(G)\geq\lambda$: whenever $A\subset X$ is a set of size $<\lambda$, by a chain condition argument there is an ordinal $\ga<\lambda$ such that the set $A$ belongs to the model obtained after the $\ga$-th stage of the iteration. Then, look at the generic compact $G$-anticliques $K_n\subset X$ for $n\in\gw$ obtained at the respective $\ga+n$-th stages of the iteration. A genericity argument shows that $A\subset\bigcup_nK_n$ must hold, and so the set $A$ is not a witness to $\kappa(G)<\lambda$.

The method of proof brings up the following question:

\begin{question}
Suppose that $G$ is a closed graph on a compact metrizable space $X$. Let $\lambda$ be a regular cardinal such that $\lambda^{\aleph_0}=\lambda$. If $\mathfrak{b}<\kappa(G)$ holds in some extension, does
$\aleph_1=\mathfrak{b}<\kappa(G)=\lambda$ hold in some extension?
\end{question}

\noindent In view of Theorem~\ref{uniquenesstheorem}, this question really asks whether the preservation of unbounded sequences of length $\gw_1$ by Borel c.c.c.\ forcings is suitably absolute in ZFC.

\section{Examples}
\label{examplesection}

It is not entirely easy to come up with closed graphs for which the concepts introduced in Section~\ref{singlesection} exhibit nontrivial interplay.  This section is devoted to a number of examples that illustrate the various fault lines.

One class of closed graphs is generated by sequences of continuous functions. If $X$ is a compact metrizable space and $f_n\colon X\to X$ are continuous functions such that for each point $x\in X$ the values $f_n(x)$ for $n\in\gw$ converge to $x$, one can consider the \emph{associated graph} $G$ on $X$ which connects points $x\neq y$ just in case there is $n\in\gw$ such that $f_n(x)=y$ or $f_n(y)=x$. The assumptions on the sequence of functions easily imply that the graph $G$ is closed. Such graphs have coloring number $\leq\aleph_1$ (just orient the edges from $x$ to $f_n(x)$) and as such cannot contain cliques of size $\aleph_2$, and by an absoluteness argument they cannot contain any perfect cliques.
 Our first example, separating the uncountable chromatic number of $G$ from the existence of perfect cliques in $G$, belongs to this class:

\begin{example}
\cite{todorcevic:examples}
Let $\langle a_n\colon n\in\gw\rangle$ be pairwise disjoint subsets of $\gw$, for each $n\in\gw$ let $g_n\in\baire$ be the increasing enumeration of $a_n\cup n$, and let $f_n\colon\cantor\to\cantor$ be the continuous function defined by $f_n(x)=x\circ g_n$. The associated closed graph $G$ has no perfect cliques and uncountable chromatic number. Thus the poset $P_G$ is c.c.c.\ but not $\gs$-centered.
\end{example}

\noindent We do not know if $\kappa(G)\leq\mathfrak{b}$ holds in ZFC for the above graph.

\begin{proof}
Let $B_n$ for $n\in\gw$ be $G$-anticliques; we must find a point $x\in\cantor\setminus\bigcup_nB_n$. For this purpose, by induction on $n\in\gw$ build binary strings $t_n\in\bintree$, numbers $m_n\in\gw$ and points $x_n, y_n\in\cantor$ such that

\begin{itemize}
\item $0=t_0\subset t_1\subset\dots$ and $t_n\subset y_n, x_n$;
\item either $B_n\cap [t_{n+1}]=0$ or else $y_n\in B_n$;
\item there is $x_n\in [t_n]$ such that for every $i\in n$, $f_{m_i}(x_n)=y_n$.
\end{itemize}

\noindent Once the induction is performed, let $x=\bigcup_nt_n\in\cantor$. By the continuity of the functions $f_m$ for $m\in\gw$ it follows from the third item of the induction hypothesis that $f_{m_i}(x)=y_i$ for all $i\in\gw$. From the second item of the induction hypothesis, it follows that $x\notin B_i$ for any $i\in\gw$ as required. 

To perform the induction, start with $t_0=0$ and $x_0\in\cantor$ arbitrary. Now suppose that $t_n, x_n$ as well as $y_i$ and $m_i$ for $i\in n$ have been found. Find a natural number $m_n$ greater than all of $m_i$ for $i\in n$ and $|t_n|$ and let $t_{n+1}=x\restriction m_n$. The construction now splits into two cases. If $B_n\cap [t_{n+1}]=0$ then let $x_{n+1}=x_n$ and $y_n=f_{m_n}(x)$; this successfully completes the induction step in this case. Otherwise, pick a point $y_n\in B_n\cap [t_{n+1}]$ and let $x_{n+1}\in\cantor$ be the point which is equal to $x_n$ except at the entries in the set $a_{m_n}$ where it satisfies $y_n=x_{n+1}\circ g_{m_n}$. This completes the induction step.
\end{proof}

Our second example is again generated by a countable collection of continuous maps. This time, it separates the countable chromatic number from the countable loose number:

\begin{example}
\label{easyexample}
For every natural number $n\in\gw$, let $f_n\colon\cantor\to\cantor$ be the function defined by $f_n(x)(i)=x(i)$ if $i\in n$ and $f_n(x)(i)=0$ if $i\notin n$. The associated closed graph $G$ on $\cantor$ has countable chromatic number, its loose number is equal to $\mathfrak{d}$ and $\kappa(G)=\mathfrak{b}$, provably in ZFC.
\end{example}

\begin{proof}
To see that the chromatic number of the graph $G$ is countable, let $A\subset\cantor$ be the countable set of all binary sequences which are eventually zero. The set $\cantor\setminus A$ is a $G$-anticlique by the definition of the graph $G$, and so $\cantor$ can be written as a countable union of (even Borel) $G$-anticliques: $\cantor=A\cup\bigcup_{x\in A}\{x\}$. For the evaluation of the other cardinal invariants, a claim will be helpful:

\begin{claim}
\label{lilclaim}
If $B\subset\cantor$ is a $G$-loose set then $\bar B\setminus A\subset\cantor$ is a closed subset of $\cantor$.
\end{claim}

\begin{proof}
 If the conclusion fails, there must be a point $x\in A$
and points $y_n\in B$ for $n\in\gw$ such that $y_n\neq x$ and $\lim_ny_n=x$. A review of the definition of the graph $G$ shows that all but finitely many points $y_n$ are $G$-connected with $x$, showing that $B$ is not $G$-loose.
\end{proof}

Now, to see that $\lambda(G)\geq\mathfrak{d}$, suppose that $\cantor=\bigcup_{i\in I}B_i$ is a union of $G$-loose sets. By Claim~\ref{lilclaim}, $\cantor\setminus A=\bigcup_{i\in I}\bar B_i\setminus A$ is a union of compact sets. Since the set $\cantor\setminus A$ is homeomorphic to the Baire space, $|I|\geq\mathfrak{d}$ immediately follows. To see that $\lambda(G)\leq\mathfrak{d}$, just write $\cantor\setminus A$ as a union of $\mathfrak{d}$ many compact sets. All of these sets are $G$-anticliques and therefore also $G$-loose sets. Adding the singletons from the countable set $A$ to the cover, we get a cover of $\cantor$ by $\mathfrak{d}$ many $G$-loose sets.

To see that $\kappa(G)\leq\mathfrak{b}$, use the fact that $\cantor\setminus A$ is homeomorphic to the Baire space again to find a set $B\subset\cantor\setminus A$ of size $\mathfrak{b}$ which cannot be covered by countably many compact subsets of $\cantor\setminus A$. By Claim~\ref{lilclaim}, it cannot be covered by countably many loose $G$-sets, in particular by countably many compact $G$-anticliques. To see that $\kappa(G)\geq\mathfrak{b}$, if $B\subset\cantor$ cannot be covered by countably many compact $G$-anticliques, then $B\cap (\cantor\setminus A)$ cannot be covered by countably many compact subsets of $\cantor\setminus A$ as desired.
\end{proof}

Another interesting family of closed graphs arises from metrics on compact metrizable spaces. Let $X$ be a compact metrizable space with a compatible metric $d$. Let $\langle r_n\colon n\in\gw\rangle$ be a sequence of positive real numbers converging to $0$. The \emph{associated graph} connects points $x, y\in X$ if $d(x, y)=r_n$ for some $n\in\gw$. In this class of closed graphs, perfect cliques are possible in zero-dimensional spaces. For example, if $d$ is the usual least difference metric on $\cantor$, then the set of all possible values of $d$ forms a sequence converging to $0$ and so even the whole space may be a clique in this case.

The most interesting representatives of the metric generated graphs are connected with the Euclidean metrics:

\begin{example}
\textnormal{\cite[Theorem 7]{komjath:list}}
Let $n\leq 3$ be a natural number, let $d$ be the Euclidean metric on $[0, 1]^n$, let $\langle r_n\colon n\in\gw\rangle$ be a sequence of positive real numbers converging to $0$, and let $G$ be the associated closed graph on $[0, 1]^n$. The graph $G$ has countable coloring number.
\end{example}

\noindent The dimensions higher than $3$ surprisingly yield an example separating the countable coloring number from the countable loose number:

\begin{example}
\textnormal{\cite{schmerl:list}}
Let $n>3$ be a natural number, let $d$ be the Euclidean metric on $X=[0, 1]^n$, let $\langle r_n\colon n\in\gw\rangle$ be a sequence of positive real numbers converging to $0$, and let $G$ be the associated closed graph on $[0, 1]^n$. The graph $G$ has uncountable coloring number but countable loose number.
\end{example}

\begin{proof}
To see that the coloring number is uncountable, choose a number $r\in\R$ such that for some $n\in\gw$ $\sqrt{2r}=r_n$ and $r<1$, and consider the sets $A_0=\{\langle x, y, 0, 0\dots\rangle\in X\colon x^2+y^2=r\}$
and $A_1=\{\langle 0, 0, x, y, 0, 0\dots\rangle\in X\colon x^2+y^2=r\}$. These are uncountable sets such that each point of $A_0$ is $G$-connected with each point in $A_1$. Such sets cannot exist in graphs of countable coloring number, say by Theorem~\ref{maintheorem} below.

To see that the loose number of $G$ is countable, note that \cite{schmerl:list} found a well-ordering $\leq$ of $X$ such that for every point $x\in X$, the set $\{y\in X\colon y\leq x\land \langle x, y\rangle\in G\}$ is bounded away from $x$. For every number $m\in\gw$, let $A_m=\{x\in X\colon\forall y\leq x\ \langle x, y\rangle\in G\to d(x, y)>2^{-m}\}$. The choice of the well-ordering $\leq$ implies that $\bigcup_mA_m=X$; thus, it will be enough to argue that each set $A_m$ is $G$-loose. To see this, suppose towards a contradiction that $z\in X$ is a point and $\langle x_i\colon i\in\gw\rangle$ is a sequence of points in $A_m$ $G$-connected to some point $z$. Let $\eps>0$ be a real number such that for all $y\leq z$ which is $G$-connected to $z$, $d(y, z)>\eps$. Then, for every $i\in\gw$, if $x_i\leq z$ then $d(x_i, z)>\eps$, and if $x_i\geq z$ then $d(x_i, z)>2^{-m}$, showing that the sequence is bounded away from the point $z$.
\end{proof}

\noindent The metric spaces of finite dimension behave quite differently in this respect than the infinite dimensional ones \cite{engelking:dimension}. The dimension fault line appears in the following example:

\begin{example}
Let $X$ be a strongly infinite-dimensional compact metrizable space and let $d$ be a compatible metric. Let $\langle r_n\colon n\in\gw\rangle$ be a sequence of positive reals converging to $0$, and let $G$ be the associated metric graph. The graph $G$ has uncountable chromatic number.
\end{example}

\noindent In fact, in the usual metrizations of the Hilbert cube for example, this graph contains perfect cliques. We do not know if it is possible to have a strongly infinite-dimensional space with a compatible metric such that the associated metric graph has no perfect cliques.

\begin{proof}
The argument starts with an auxiliary claim:

\begin{claim}
Let $K\subset X$ be a set, $\eps>0$, and $B\subset X$ be a $G$-anticlique. There is an open set $O\subset X$
such that $K\subset O$, every point of $O$ is within $\eps$-distance from some point in $K$, and the boundary of $O$ has empty intersection with $B$.
\end{claim}

\begin{proof}
Replacing $K$ with its closure we may assume that $K$ is compact. Every point $x\in K$ is contained either (1) in some open ball of radius $<\eps/2$ whose closure contains no elements of $B$, or (2) in some open ball of radius $<\eps/2$ whose center is in $B$ and whose radius belongs to the set $\{r_n\colon n\in\gw\}$. By a compactness argument, the whole set $K$ is covered by finitely many balls of this type. Let $O$ be the union of the finitely many balls. It is clear that $K\subset O$ and every point of $O$ is within $\eps$-distance from some point in $X$. Finally, the boundary of the set $O$ is a subset of the union of the boundaries of the finitely many balls in the union, and none of them contain any elements of the anticlique $B$: in case (1) this occurs because the closure of the whole ball contains no elements of the set $B$, and in case (2), this occurs because $B$ is a $G$-anticlique and the points on the boundary of the ball are $G$-related to the center of the ball which belongs to $B$.
\end{proof}

Now, suppose that $B_n$ for $n\in\gw$ are $G$-anticliques; we must produce a point $x\in X\setminus\bigcup_nB_n$. Use the infinite dimensionality of the space $X$ to find an essential sequence $\langle K_n^0, K_n^1\colon n\in\gw\rangle$; that is, $K_n^0, K_n^1$ are disjoint nonempty subsets of $X$ for each $n$, and whenever $O_n\subset X$ for $n\in\gw$ are open sets such that $K_n^0\subset O_n$ and $\bar O_n\cap K_n^1=0$ then the intersection of the boundaries of the sets $O_n$ is nonempty.
Now, use the claim to find, for each number $n\in\gw$, an open set $O_n\subset X$ such that $K_n^0\subset O_n$ and $\bar O_n\cap K_n^1=0$ and the boundary of the set $O_n$ has empty intersection with the anticlique $B_n$. By essentiality, the intersection of the boundaries of the sets $O_n$ is nonempty, and any point in it belongs to $X\setminus\bigcup_nB_n$ as desired.
\end{proof}

Finally, we owe the reader an example showing that the conclusion of Theorem~\ref{main1theorem} fails if one considers graphs only slightly more complicated than closed:

\begin{example}
Let $G$ be the graph on $\cantor$ connecting sequences $x, y$ if they differ on at most finitely many entries. The graph $G$ is locally countable and therefore has countable coloring number. At the same time, $\kappa(G)\leq\mathfrak{b}$ holds in ZFC.
\end{example}

\noindent It is well-known that each measurable $G$-anticlique must have zero $\mu$-mass, where $\mu$ is the Haar probability measure on $\cantor$. Thus, no set of positive outer $\mu$-mass can be covered by countably many $G$-anticliques. This shows that $\kappa(G)\leq\mathtt{non(null)}$ holds in ZFC, in particular $\kappa(G)<\mathfrak{b}$ is consistent with ZFC.

\begin{proof}
Let $A\subset\cantor$ be the countable set of binary sequences which are eventually zero, so that $\cantor\setminus A$ is homeomorphic to the Baire space. Let $F\subset\cantor\setminus A$ be a set of size $\mathfrak{b}$ which cannot be covered by countably many compact subsets of $\cantor\setminus A$. We claim that $F$ cannot be covered by countably many $G$-anticliques.

Indeed, if $K\subset\cantor$ is a compact $G$-anticlique, then it intersects the set $A$ in at most one point, and so $K\setminus A$ is a union of countably many compact sets. It follows that if $F$ were covered by countably many $G$-anticliques, it would be covered by countably many compact subsets of $\cantor\setminus A$, which is impossible by the choice of the set $F$.
\end{proof}

\section{A dichotomy}
\label{dichotomysection}

The purpose of this section is to characterize those analytic graphs on Polish spaces for which the coloring number is countable. The evaluation of the coloring number for finite or infinite graphs is a fairly involved business, see \cite{alon:list}. Even in the case of a closed or analytic graph, the orientations witnessing the coloring number are typically obtained through a heavy use of the axiom of choice. However, the existence of such an orientation can be characterized by a simple formula. It turns out that there is a minimal analytic graph of uncountable coloring number, which is in addition a clopen graph on a (noncompact) $\gs$-compact Polish space:

\begin{definition}
Let $Y$ be the Polish space which is the disjoint union of $2^{<\gw}$, viewed as a discrete space, and $\cantor$ with its usual topology. $G_1$ is the graph on $Y$ given by $G_1=\{\{y\restriction n, y\}\colon y\in\cantor, n\in\gw\}$.
\end{definition}

\begin{theorem}
\label{maintheorem}
Let $G$ be an analytic graph on a Polish space $X$. The following are equivalent:

\begin{enumerate}
\item the coloring number of $G$ is countable;
\item for every countable set $a\subset X$, the set $\{x\in X\colon\exists^\infty y\in a\ x\mathrel{G}y\}$ is countable;
\item there is no continuous injective homomorphism of $G_1$ to $G$.
\end{enumerate}

\noindent If there is a proper class of Woodin cardinals then the conclusion holds for all universally Baire graphs.
\end{theorem}

\begin{proof}
We will deal with the case of analytic graphs; the case of universally Baire graphs under a large cardinal assumption is left to the reader, as it needs no additional tricks.

The (1)$\to$(2) direction does not use any definability assumptions on the graph $G$. If (1) holds and $a\subset X$ is a countable set, let $o$ be an orientation of $G$ in which the outflow of every point is finite, and let $M$ be a countable elementary submodel of a large structure containing both $a$ and $o$ as elements. Since the set $a$ is countable, it is also a subset of $M$ by elementarity, and it will be enough to show that for every point $x\in X\setminus M$, $x$ is connected with only finitely many elements of $X\cap M$. Indeed, if $x$ were connected with infinitely many points in $M$, then there would be a point $y\in M$ which is not in the finite $o$-outflow of $x$ and is connected with $x$. It follows that $x$ is in the outflow of $y$; but, as the outflow of $y$ is finite, this implies that $x\in M$ by the elementarity of the model $M$. A contradiction.

For the (2)$\to$(1) direction, assume that (2) holds and prove the following claim:

\begin{claim}
\label{bclaim}
For every submodel $M$ of a large enough structure containing the graph $G$ and the space $X$ as elements, for every point $x\in X\setminus M$, $x$ is $G$-connected with only finitely many elements of the model $M$.
\end{claim}

\noindent Note that there is no cardinality restriction on the submodel $M$.

\begin{proof}
Suppose that $M$ is an elementary submodel of some large structure containing the space $X$ and the graph $G$. Let $H\subset \coll(\gw, X\cap M)$ be a filter generic over $V$. One can then form the model $M[H]$, since $\coll(\gw, X\cap M)=\coll(\gw, X)\cap M$. Comparison of the models concerned yields the following:

\begin{itemize}
\item $M[H]\cap V=M$. This follows from the genericity of the filter $H$.
\item $M[H]\models$ for every countable set $a\subset X$, the set $\{x\in X\colon\exists^\infty y\in a\ x\mathrel{G}y\}$ is countable. This is because the given statement is coanalytic, true in $V$ and one can apply the Mostowski absoluteness to the wellfounded model $M[H]$.
\item $M[H]\models  X\cap M$ is countable, and by the previous item there is an $\gw$-sequence $z\in X^\gw$ in $M[H]$ such that
$M[H]\models \forall x\notin\rng(z)\ \{y\in X\cap M\colon x\mathrel G y\}$ is finite;
\item $V[H]\models \forall x\notin\rng(z)\ \{y\in X\cap M\colon x\mathrel G y\}$ is finite, since this property of the countable set $X\cap M$ and the sequence $z$ is coanalytic, and by the Mostowski absoluteness it can be transfered from $M[H]$ to $V[H]$.
\end{itemize}

Now, if a point $x\in X$ in $V\setminus M$ is an arbitrary point, then $x\notin M[H]$ by the first item, so $x\notin\rng(z)$, and by the last item $x$ is connected to only finitely many elements of $X\cap M$ as required in the claim.
\end{proof}

Now, by induction on the cardinality of a set $A\subset X$ argue that the graph $G$ restricted to $A$ can be oriented so that the outflow of every point in $A$ is finite. This is immediate if $|A|=\aleph_0$. Now suppose that $|A|=\kappa$ and the statement has been proved for all sets of size $<\kappa$. Choose a continuous increasing sequence $\langle M_\ga\colon\ga\in\mathtt{cf}(\kappa)\rangle$ of elementary submodels of a large structure such that $f, A\in M_0$, $|M_\ga|<\kappa$ for every $\ga$, and $A\subset\bigcup_\ga M_\ga$. Use the inductive assumption to find an orientation $o_\ga$ of the graph $G\restriction A\cap M_\ga$ for each ordinal $\ga\in\gb$ such that the outflow of any node is finite. Define an orientation $o$ of $G\restriction A$ by orienting an edge $\langle x, y\rangle$ towards $y$ just in case either the smallest ordinal $\ga$ for which $y$ appears in $M_\ga$ is smaller that the smallest ordinal $\ga$ for which $x$ appears in $M_\ga$, or in the case that these two ordinals are equal to some $\ga$, then $\langle x, y\rangle$ is oriented in the same way by $o$ as it is in $o_\ga$. Claim~\ref{bclaim} immediately implies that the $o$-outflow of any node in $A$ is finite as required.

Now, the negation of (3) implies the negation of (2) since in the graph $G_1$, all the non-isolated points of the underlying space are connected to infinitely many elements of the countable set of isolated points. To see how the negation of (2) implies the negation of (3), fix a countable set $a\subset X$ such that the set $B=\{x\in X\colon \exists^\infty y\in a\ \colon x\mathrel G y\}$ is uncountable. Since the set $B$ is analytic, it contains a perfect subset $C\subset B$ disjoint from $a$. Now, by induction on $t\in\bintree$ build perfect sets $C_t\subset C$ and points $x_t\in a$ such that

\begin{itemize}
\item $t\subset s$ implies $C_s\subset C_t$, the set $C_t$ has diameter $2^{-|t|}$ in some fixed complete compatible metric on $X$, and it is either disjoint from or a subset of the $|t|$-th basic open subset of $X$ in some fixed enumeration of a countable topology base for $X$;
\item for each $t\in\bintree$, the sets $C_{t^\smallfrown 0}$ and $C_{t^\smallfrown 1}$ are pairwise disjoint;
\item the points $x_t\in a$ are pairwise distinct and all points in $C_t$ are $G$-connected with $x_t$.
\end{itemize}

\noindent To perform the induction step, suppose that the sets $C_t$ and points $x_t$ for $t\in 2^{\leq n}$ have been constructed.
First, use standard arguments to find sets $C'_s$ for each $s\in 2^{n+1}$ which satisfy the first two items above. Now, enumerate $2^{n+1}$ as $\{s_i\colon i\in j\}$ and by subinduction on $i$ find points $x_{s_i}\in a$ such that they are pairwise distinct and also distinct from the points $x_t$ for $t\in 2^{\leq n}$, and such that there are uncountably many elements of the set $C'_{s_i}$ which are $G$-connected to the point $x_{s_i}$. The subinduction is easy to perform given the fact that the set $a$ is countable and every element of $C'_{s_i}$ is connected to infinitely many of its members. In the end, use the perfect set theorem to find a perfect set $C_{s_i}\subset C'_{s_i}$ of points connected to $x_{s_i}$. This completes the induction step.

Once the induction has been performed, consider the function $f\colon Y\to X$ given by the following description. If $t\in\bintree$ then $f(t)=x_t$ and if $y\in\cantor$ then $f(y)$ is the unique point in $\bigcap_nC_{y\restriction n}$. It is clear from the induction assumptions that $f$ is a continuous injection from $Y$ to $X$ which is moreover a homomorphism of the graph $G_1$ to $G$.
\end{proof}

\begin{corollary}
\label{listcorollary}
For an analytic graph $G$ on a Polish space $X$, the coloring number of $G$ is countable if and only if the list-chromatic number of $G$ is countable.
\end{corollary}

Here, the list-chromatic number of $G$ \cite{hajnal:chromatic} is the smallest cardinal $\kappa$ such that for every function $F$ which assigns each point $x\in X$ a set of size $\geq\kappa$, there is a function $f$ which assigns each point $x$ and element of $F(x)$ such that for any two $G$-connected points $x_0, x_1\in X$, $f(x_0)\neq f(x_1)$ holds. Komj{\' a}th \cite{komjath:list1}showed that consistently, the list-chromatic number of infinite graphs can be equal to the coloring number, and consistently, for all graphs of size $\aleph_1$, if the chromatic number is countable then so is the list-chromatic number. The theorem shows that no such antics appear in the case of analytic graphs and countable list-chromatic number.

\begin{proof}
Since for any graph $G$, the list-chromatic number is not greater than the coloring number, it is enough to show that the graph $G_1$ on the space $Y=\bintree\cup\cantor$ has uncountable list-chromatic number. For this, let $E$ be a (Borel) bijection between $\cantor$ and the set of all maps $g\colon \bintree\to\gw$ such that $\forall t\in\bintree\ g(t)>|t|$. Let $F$ be the function which to each $t\in\bintree$ assigns the set $\{n\in\gw\colon |t|<n\}$, and to each point $y\in\cantor$ assigns the infinite set $\{E(y)(y\restriction n)\colon n\in\gw\}$.
The function $F$ stands witness to the fact that the list-coloring number of the graph $G_1$ is uncountable.

To see this, if $f$ is a function on the space $Y$ which for each $y\in Y$ selects an element of $F(y)$, then there must be a point $y\in\cantor$ such that $E(y)=f\restriction\bintree$. But then, there must be $n\in\gw$ such that $f(y)=E(y)(y\restriction n)=f(y\restriction n)$. Since $y\restriction n$ is $G_1$-connected to $y$, this completes the proof.
\end{proof}

The equivalence of items (1) and (2) in Theorem~\ref{maintheorem} fails for graphs which are not definable, as the following example shows:

\begin{example}
There is a graph $G$ on $\gw_1$ such that the coloring number of $G$ is uncountable, while for every countable set $a\subset X$, the set $\{x\in X\colon\exists^\infty y\in a\ x\mathrel{G}y\}$ is countable.
\end{example}

\begin{proof}
For every countable limit ordinal $\ga$ choose a set $c_\ga\subset\ga$ which is cofinal in it and of ordertype $\gw$. The graph $G$ on $\gw_1$ consists of all pairs $\{\gb, \ga\}$ such that $\ga$ is a limit ordinal and $\gb\in c_\ga$.

First of all, if $a\subset\gw_1$ is a countable set and $\ga$ is any successor ordinal larger than $\sup(a)$, then no ordinal $\gb\in\gw_1$ is connected with infinitely many elements of $\ga$, verifying that the set $\{\gb\in\gw_1\colon\exists^\infty \gg\in a\ \gb\mathrel{G}\gg\}$ is countable. At the same time, the coloring number of the graph $G$ is uncountable: if $o$ is some orientation of $G$, consider the regressive function $f$ on $\gw_1$ which assigns to each limit ordinal $\ga$ an element of $c_\ga$ which does not belong to the $o$-outflow of $\ga$.
Apply Fodor's theorem to find a stationary set $S\subset\gw_1$ and an ordinal $\gb$ such that $f(\ga)=\gb$ for all $\ga\in S$. Now, if $\ga\in S$ is any ordinal which is not in the outflow of $\gb$, there is no way of orienting the edge $\{\gb, \ga\}\in G$ in a way consistent with the assumptions.
\end{proof}

It is immediate from Theorem~\ref{maintheorem} that the concept of countable coloring number is \pioneoneonsigmaoneone\ in the following sense \cite[Section 29.E]{kechris:classical}. If $X$ is a Polish space and $A\subset\baire\times X\times X$ is an analytic set whose vertical sections are symmetric and reflexive relations, then the set $\{y\in\baire\colon A_y$ has countable coloring number$\}$ is coanalytic. The final computation in this paper shows that this is an optimal complexity bound:

\begin{proposition}
The collection of closed graphs on $\cantor$ with countable coloring number is a complete coanalytic subset of the space $K((\cantor)^2)$.
\end{proposition}

\begin{proof}
It is enough to work with any zero-dimensional compact space in place of $\cantor$ as such a space is homeomorphic to a closed subspace of $\cantor$. Let $X=\bigcup_n 2^n\times 2^n\cup (\cantor\times\cantor)$ and its topology is generated by sets $O_{t,s}$ where for some $n\in\gw$ $t, s\in 2^n$ and $O_{t, s}=\{\langle u, v\rangle\in X\colon t\subset u, s\subset v\}$. For every tree $T\subset\bintree$ consider the graph $G_T$ connecting $\langle t, s\rangle\in X$ with $\langle u, v\rangle\in\cantor\times\cantor$ if $t\in T$, the last entry on $t$ is $1$, and $t\subset u, s\subset v$. It is not difficult to check that the map $T\mapsto G_T$ is a continuous map from the space of all trees to $K(X^2)$. I claim that the coloring number of $G_T$ is countable just in case the tree $T$ has no infinite branch with infinitely many unit entries. This will complete the proof of the theorem since the set of binary trees which have an infinite branch with infinitely many unit entries is a complete analytic set \cite[Exercise 27.3]{kechris:classical}.

Indeed, if $T$ has no infinite branch with infinitely many unit entries, then orient edges in $G_T$ so that they point from elements of $\cantor\times\cantor$ to the pairs of finite binary sequences connected with them. The lack of infinite paths in $T$ shows that the outflow of every node in this orientation of the graph $G_T$ is finite, and so the coloring number of $G_T$ is countable. On the other hand, if $x\in\cantor$ is an infinite path through the tree $T$ which contains infinitely many units, then every point of the form $\langle x, y\rangle$ for $y\in\cantor$ is connected with infinitely many nodes in the countable collection $\bigcup_n 2^n\times 2^n$, showing that the coloring number of the graph $G_T$ is uncountable by Theorem~\ref{maintheorem}.
\end{proof}

\bibliographystyle{plain} 
\bibliography{odkazy,zapletal}

\end{document}